\documentclass[a4paper]{article}

\usepackage[english]{babel}
\usepackage[utf8x]{inputenc}
\usepackage[T1]{fontenc}
\usepackage{amssymb}
\usepackage{amsthm}
\usepackage{makeidx}
\usepackage{color}
\usepackage{fullpage}
\usepackage{tikz}
\usepackage{mathtools}
\usepackage[all]{xy}
\usepackage[colorlinks=true, allcolors=blue]{hyperref}
\usepackage{comment}
\usepackage{lipsum}
\usepackage{authblk}

\usepackage[a4paper,top=3cm,bottom=2cm,left=3cm,right=3cm,marginparwidth=1.75cm]{geometry}

\usepackage{amsmath}
\usepackage{graphicx}
\usepackage[colorinlistoftodos]{todonotes}
\usepackage[colorlinks=true, allcolors=blue]{hyperref}
\usepackage[all]{xy}

\newtheorem{defi}{Definition}[section]
\newtheorem{example}[defi]{Example}
\newtheorem{teo}[defi]{Theorem}
\newtheorem{coro}[defi]{Corollary}

\newtheorem{pro}[defi]{Proposition}
\newtheorem{rmk}[defi]{Remark}

\newcommand{\GL}{\operatorname{GL}}
\newcommand{\End}{\operatorname{End}}
\newcommand{\Perm}{\operatorname{Perm}}
\newcommand{\Q}{\mathbb{Q}_3}

\title{Induced Hopf Galois Structures and their Local Hopf Galois Modules}
\author{Daniel Gil-Mu\~noz, Anna Rio}

\date{\vspace{-5ex}}

\begin{document}

\newcommand{\Addresses}{
\bigskip
\footnotesize

D. Gil-Muñoz, \textsc{Departament de Matem\`atiques, Universitat Polit\`ecnica de Catalunya, Edifici Omega, Jordi Girona, 1-3, 08034 Barcelona}\par\nopagebreak
\textit{E-mail address}: \texttt{daniel.gil.munoz@upc.edu}

\medskip

A. Rio, \textsc{Departament de Matem\`atiques, Universitat Polit\`ecnica de Catalunya, Edifici Omega, Jordi Girona, 1-3, 08034 Barcelona}\par\nopagebreak
\textit{E-mail address}, \texttt{ana.rio@upc.edu}

}

\maketitle

\begin{abstract}
The regular subgroup determining an induced Hopf Galois structure for a Galois extension $L/K$ is obtained as the direct product of the corresponding regular groups of the inducing subextensions. We describe here the associated Hopf algebra and Hopf action of an induced structure and we prove that they are obtained by tensoring the corresponding inducing objects. In order to deal with their associated orders we develop a general method to compute bases and free generators in terms of matrices coming from representation theory of Hopf modules. In the case of an induced Hopf Galois structure it allows us to decompose the associated order, assuming that inducing subextensions are arithmetically disjoint.
\end{abstract}

\tableofcontents

\section{Introduction}

A finite extension of fields $L/K$ is said to be Hopf Galois if there is a $K$-Hopf algebra $H$ and a $K$-linear action of $H$ in $L$ which endows $L$ with an $H$-module structure $H\to \End_K(L)$, such that the induced map $j\colon L\otimes_K H\longrightarrow\mathrm{End}_K(L)$ is bijective. In that case, the pair formed by the Hopf algebra and the Hopf action is said to be a Hopf Galois structure of $L/K$. For simplicity, we may say that $L/K$ is $H$-Galois. This notion was introduced by Chase and Sweedler in their book \cite{chasesweedler} and it generalizes the notion of Galois extension of fields since the group algebra of the Galois group together with the Galois action on $L$ provides a Hopf Galois structure. 

Although Hopf Galois structures are difficult to compute in general, the ones of separable extensions can be labeled by objects of group theory thanks to a theorem of Greither and Pareigis. This result was introduced in their paper \cite{greitherpareigis}. If $L/K$ is a separable extension, let $\widetilde{L}$ be its normal closure, $G=\mathrm{Gal}(\widetilde{L}/K)$, $G'=\mathrm{Gal}(\widetilde{L}/L)$ and $X=G/G'$. Let $\lambda \colon G\longrightarrow\mathrm{Perm}(X)$ be the left translation map, namely the embedding given by the group action of $G$ on $X$ by left translation.

\begin{teo}[Greither-Pareigis] Hopf Galois structures of $L/K$ are in one-to-one correspondence with regular (i.e, simply transitive) subgroups of $\mathrm{Perm}(X)$ normalized by $\lambda(G)$. Moreover, if $N$ is some such subgroup, the corresponding Hopf Galois structure is given by the $K$-Hopf algebra $\widetilde{L}[N]^G$ and its action on $L$ defined by $$\left(\sum_{i=1}^rc_in_i\right)\cdot x=\sum_{i=1}^rc_in_i^{-1}(\,\overline{1}_G\,)(x)$$
\end{teo} 
\noindent See \cite[Theorem 6.8]{childs}. 

If a Hopf Galois structure corresponds to a group $N$, its type is defined as the isomorphism class of $N$. 

Hopf Galois theory can be used to generalize Galois module theory, giving rise to a Hopf Galois module theory. In the approach introduced by Leopoldt, Galois module theory studies the structure of the valuation ring $\mathcal{O}_L$ of a Galois extension $L/K$ of local fields over its associated order $\mathfrak{A}_{K[G]}$, defined as the maximal $\mathcal{O}_K$-order in $K[G]$ that acts on $\mathcal{O}_L$. If $L/K$ is $H$-Galois, the associated order of $\mathcal{O}_L$ in $H$ is defined as $$\mathfrak{A}_H=\{\alpha\in H\,|\,\alpha\cdot x\in\mathcal{O}_L\hbox{ for every }x\in\mathcal{O}_L\}.$$ In the context of the existence of a normal integral basis, it is interesting to determine the freeness of $\mathcal{O}_L$ as $\mathfrak{A}_H$-module as $H$ runs through the different Hopf Galois structures of $L/K$.

In this paper we consider Hopf Galois structures of a semidirect extension $L/K$ of fields, that is, a Galois extension whose Galois group $G$ is a semidirect product. We will write $G=J\rtimes G'$ with $J$ a normal subgroup of $G$ and $E=L^{G'}$, $F=L^J$ and we will focus on induced Hopf Galois structures. These are structures built from a Hopf Galois structure on $E/K$ and a Hopf Galois structure on $L/E$. This notion was introduced by Crespo, Rio and Vela in their paper \cite{cresporiovela}. 
The type of an induced Hopf Galois structure is the direct product of the types of the corresponding Hopf Galois structures of $L/E$ and $E/K$. As a consequence, the regular subgroup $N$ which determines such a Hopf Galois structure is a direct product $N=N_1\times N_2$. 

Our first results show that certain properties are directly related to the type, since we have a complete analogy with the situation of classical Galois extensions whose Galois group is a direct product of subgroups. For such extensions we can use standard Galois theory to prove the following:
\begin{pro}\label{progaloisinduced}
Let $L/K$ be a finite Galois extension with Galois group $G$ and assume that $G$ can be written as a direct product $G=J\times G'$. Let $E=L^{G'}$ and  $F=L^{J}$. Then:
\begin{itemize}
    \item[1.] $E/K$ and $F/K$ are Galois extensions.
    \item[2.] $L=EF$ and $E\cap F=K$.
    \item[3.] $E/K$ and $F/K$ are linearly disjoint, namely the canonical map $E\otimes_K F\longrightarrow EF$ is a $K$-isomorphism.
    \item[4.] $K[G]=K[J]\otimes_K K[G']$.
    \item[5.] The Galois action of $K[G]$ on $L$ is the Kronecker product of the Galois actions of $K[J]$ on $E$ and $K[G']$ on $F$.
\end{itemize}
\end{pro}

In the setting of Hopf Galois theory, Proposition \ref{progaloisinduced} gives information about the classical Galois structure of $L/K$. We shall check that induced Hopf Galois structures of Galois extensions $L/K$ with Galois group of the form $G=J\rtimes G'$ present a similar behaviour. 

We note that the condition that $G'$ has a normal complement $J$ is equivalent to the condition (from \cite{greitherpareigis}) that $E/K$ is an almost classical Hopf Galois extension. On the other hand, $L/E$ is Galois and therefore Hopf Galois. By \cite[Theorem 3]{cresporiovela}, every choice of Hopf Galois structures of $E/K$ and $L/E$ determines a Hopf Galois structure of $L/K$. On the other hand, since $J$ is normal $F/K$ is Galois and hence a Hopf Galois extension. We shall study in \ref{secinducedhopfalg} the relationship between structures of $L/E$ and $F/K$.

\begin{teo}\label{firstmaintheorem}
Let $L/K$ be a finite Galois extension with Galois group $G=J\rtimes G'$. Let $E=L^{G'}$ and $F=L^J$. Then:
\begin{itemize}
    \item[1.] $E/K$ and $F/K$ are Hopf Galois extensions.
    \item[2.] $L=EF$ and $E\cap F=K$.
    \item[3.] $E/K$ and $F/K$ are linearly disjoint.
\end{itemize}

Let $E/K$ be $H_1$-Galois and let $L/E$ be $H_2$-Galois. We consider the corresponding induced Hopf Galois structure of $L/K$ . Let $H$ be its associated Hopf algebra. Then:
\begin{itemize}
    \item[4.] $H=H_1\otimes_K\overline{H}$, where $\overline{H}$ is the Hopf algebra of the Hopf Galois structure of $F/K$ such that $\overline{H}\otimes_K E=H_2$ (see Proposition \ref{proptensorcorresp}).
   
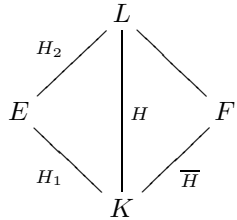
\[
\xymatrix{
& \ar@{-}[dl]_{H_2} L \ar@{-}[dd]^{H} \ar@{-}[dr] & \\
E \ar@{-}[dr]_{H_1} & & \ar@{-}[dl]^{\overline{H}} F \\
& K &}
\]
    \item[5.] The Hopf action of $H$ on $L$ is the Kronecker product of the Hopf actions of $H_1$ on $E$ and $\overline{H}$ on $F$.
\end{itemize}
\end{teo}

The second and third items are again straightforward by using basic Galois theory. We shall prove the remaining assertions in Section \ref{sectioninduced}.

In Sections \ref{sectionintegral} and \ref{secfreeness} we deal with the Hopf Galois module structure at the integer level. $K$ will be the quotient field of a principal ideal domain $\mathcal{O}_K$, $L/K$ will be a finite separable Hopf Galois extension and $\mathcal{O}_L$ will be the integral closure of $\mathcal{O}_K$ in $L$. As in linear representation theory, we study Hopf actions trough matrices and we show how to use matrices to obtain a basis for the associated order and to check for the existence of a free generator of $\mathcal{O}_L$ as module over the associated order. 

When we restrict ourselves to induced Hopf Galois structures, we can ask whether the induced Hopf Galois module structure of $\mathcal{O}_L$ can be described from the Hopf Galois module structures of $\mathcal{O}_E$ and $\mathcal{O}_F$. With regard to the classical Galois module structure of $\mathcal{O}_L$ when $L/K$ is a direct extension (i.e, with Galois group isomorphic to a direct product), Byott and Lettl included in \cite{byottlettl} the following:
\begin{pro} Let $K$ be the quotient field of a Dedekind domain $\mathcal O_K$ and let $E/K$, $F/K$ be finite Galois extensions. Put $L=EF$ and suppose that $E/K$ and $F/K$ are arithmetically disjoint. Then:
\begin{itemize}
    \item[1.] $\mathfrak{A}_{L/F}=\mathfrak{A}_{E/K}\otimes_{\mathcal{O}_K}\mathcal{O}_{F}$ and $\mathfrak{A}_{L/K}=\mathfrak{A}_{E/K}\otimes_{\mathcal{O}_K}\mathfrak{A}_{F/K}$.
    \item[2.] If there exists some $\gamma\in\mathcal{O}_E$ with $\mathcal{O}_E=\mathfrak{A}_{E/K}\cdot\gamma$, then $\mathcal{O}_L=\mathfrak{A}_{L/F}\cdot(\gamma\otimes1)$. \\ If there also exists $\delta\in\mathcal{O}_{F}$ with $\mathcal{O}_F=\mathfrak{A}_{F/K}\cdot\delta$, then $\mathcal{O}_L=\mathfrak{A}_{L/K}\cdot(\gamma\otimes\delta)$.
\end{itemize}
\end{pro}
\noindent See \cite[Lemma 5]{byottlettl}.

Recall that $E/K$ and $F/K$ are said to be arithmetically disjoint over $K$ if 
$\mathcal{O}_{EF}=\mathcal{O}_E\otimes_{\mathcal{O}_K}\mathcal{O}_F$, or equivalently, if their discriminants are coprime and $E$ is linearly disjoint from $F$ over $K$. The decomposition turns out to be very useful in order to study the Galois module structure of $\mathcal{O}_{EF}$ in terms of $\mathcal{O}_E$ and $\mathcal{O}_F$.

Going back to the semidirect case, if we keep the condition that $E/K$ and $F/K$ are arithmetically disjoint, we are able to obtain analogous results for induced Hopf Galois structures. We will describe the induced Hopf Galois module structure of $\mathcal{O}_L$ by considering the corresponding module structures of $\mathcal{O}_E$ and $\mathcal{O}_F$ as follows.

\begin{teo}\label{secondmaintheorem} 
Let $K$ be the quotient field of a principal ideal domain $\mathcal{O}_K$, $L/K$ a finite separable Hopf Galois extension and $\mathcal{O}_L$ the integral closure of $\mathcal{O}_K$ in $L$.
Assume that the structure is an induced one and its Hopf algebra is $H=H_1\otimes_K\overline{H}$. If $E/K$ and $F/K$ are arithmetically disjoint then the following statements hold:
\begin{itemize}
    \item[1.] $\mathfrak{A}_H=\mathfrak{A}_{H_1}\otimes_{\mathcal{O}_K}\mathfrak{A}_{\overline{H}}$.
    \item[2.] If $\mathcal{O}_E$ is $\mathfrak{A}_{H_1}$-free and $\mathcal{O}_F$ is $\mathfrak{A}_{\overline{H}}$-free, then $\mathcal{O}_L$ is $\mathfrak{A}_H$-free. Moreover, an $\mathfrak{A}_H$-generator of $\mathcal{O}_L$ is the product of an $\mathfrak{A}_{H_1}$-generator of $\mathcal{O}_E$ and an $\mathfrak{A}_{\overline{H}}$-generator of $\mathcal{O}_F$.
\end{itemize}
\end{teo}

Since $E/K$ and $F/K$ are linearly disjoint, it is well known that for every Hopf Galois structure $H_1$ of $E/K$, $H_1\otimes_KF$ is a Hopf Galois structure of $L/F$. We will also describe the module structure of $\mathcal{O}_L$ in this Hopf Galois structures in the following way:
\begin{teo}\label{thirdmaintheorem} Under the same hypothesis,
\begin{itemize}
    \item[1.] $\mathfrak{A}_{H_1\otimes_K F}=\mathfrak{A}_{H_1}\otimes_{\mathcal{O}_K}\mathcal{O}_F$.
    \item[2.] If $\mathcal{O}_E$ is $\mathfrak{A}_{H_1}$-free, then $\mathcal{O}_L$ is $\mathfrak{A}_{H_1\otimes_K F}$-free. 
    \end{itemize}
\end{teo}

\section{Hopf Galois linear representations}\label{secmatrixaction}

Let $A$ be an $n$-dimensional $K$-algebra and let $V$ be a $K$-vector space. A linear representation of $A$ in $V$ is a $K$-algebra
homomorphism $\rho: A → \mathrm{End}_K(V)$  or, by choosing a $K$-basis of V, a $K$-algebra
homomorphism $\rho:A → \mathrm{Mat}_n(K)$. 
This is the notion of $A$-module that we have mentioned in the introduction. 
The tensor product or Kronecker product of two representations 
$\rho^1: A_1 → \mathrm{End}_K(V_1)$ and 
$\rho^2: A_2 → \mathrm{End}_K(V_2)$
is 
$$
\rho^1\otimes \rho^2: A_1\otimes_K A_2 → \mathrm{End}_K(V_1\otimes V_2) 
$$
defined by $$\rho^1\otimes \rho^2(a_1\otimes a_2)(v_1\otimes v_2)=
\rho^1(a_1)(v_1)\otimes \rho^2(a_2)(v_2).
$$
If we have a finite field extension $L/K$,
then $L$ is a $K$- algebra and we put
$1:L\to \mathrm{End}_K(L)$ for the representation given by multiplication in $L$.

In this section we consider Hopf Galois structures from the point of view of representations,
 having in mind that when we consider induced Hopf Galois structures we will deal with the  Kronecker product of actions and also, since the Hopf action appears in the definition of the associated order, that matrices become useful for explicit computations.

Let $L/K$ be a Galois extension of degree $n$ with Galois group $G$. Since $G$ acts in $L$ we have a Galois representation $\rho:G\to \mathrm{Aut}_K(L)$. Observe that a Galois representation gives a $K[G]$-module structure, that is, a linear representations of the Hopf algebra $K[G]$, since it 
can be extended to a $K$-linear representation $\rho\colon K[G]\longrightarrow \End_K(L).$
This representation provides full information about the action of $K[G]$ on $L$. In fact, $L/K$ being Galois is equivalent to $(1, \rho):L\otimes_K K[G]\to \End_K(L).$ being an isomorphism.

In a complete analogy, a Hopf-Galois structure for a field extension $L/K$ of degree $n$ corresponds to the existence 
of a $K$-Hopf algebra $H$ and a linear representation 
$\rho_H: H → \mathrm{End}_K(L)$ such that 
$j=(1,\rho_H):L\otimes_K H\to  \mathrm{End}_K(L)$
is an isomorphism. 
Recall that in the separable case, Greither and Pareigis show that $L\otimes_K H=L[N]$ ($H$ is an $L$-form of $K[N]$) with $N$ a group of order $n$. 
 
The Hopf action of $H$ in $L$ gives rise to a $K$-bilinear map
$$
\begin{array}{ccc}
     H\times L&\to& L  \\
     (h,x)&\mapsto& h\cdot x
\end{array}
$$
with values in $L$. Therefore, given a  $K$-basis $\{w_i\}_{i=1}^{n}$ of $H$ and a $K$-basis $\{\gamma_j\}_{j=1}^{n}$ of $L$, the action can be 
described through the $n\times n$ matrix $\left(w_i\cdot\gamma_j\right)_{ij}$
with entries in $L$. We will write these elements in the chosen basis of $L$ and
we will usually represent this matrix in a table in order to make clear the bases under consideration. For $1\leq i\leq n$, the $i$-th row of the table corresponds to the matrix $\rho_H(w_i)$ representing $w_i$.

\begin{example}\label{matrixgalois}
Let $L/K$ be a finite degree $n$ Galois extension and let $G$ be its Galois group. 
Let us consider the classical Galois representation
$\rho\colon K[G]\longrightarrow \End_K(L).$

Let us write $G=\{\sigma_1,...,\sigma_{n}\}$ and take these elements as  $K$-basis for $K[G]$. On the other hand, we fix a normal basis $\{\alpha_j=\sigma_j(\alpha)\}_{j=1}^{n}$ for $L$. 
We have $\sigma_i(\alpha_j)=\sigma_i\sigma_j(\alpha)$ and
via the regular representation of $G$ we can identify $\sigma_i$ with a permutation of the set $\{1,\dots,n\}$, namely
$\sigma_i(j)=k$ if $\sigma_i\sigma_j=\sigma_k$.
With this identification, the action of $K[G]$ on $L$ is the regular representation of the group $G$.
Hence, for every $1\leq i\leq n$, $\rho(\sigma_i)$ is
an $n\times n$ permutation matrix. 
\end{example}

\begin{example}\label{matrixgaloiscyclic}
In the previous example, assume that $G\cong C_n$. Let $\sigma$ be a generator of $G$ and take $\{\mathrm{Id},\sigma,...,\sigma^{n-1}\}$ as $K$-basis of $K[G]$. Then, the action of $K[G]$ on $L$ is given by:

\begin{equation*}\label{table3}
\begin{tabular}{c|cccc} & \begin{tabular}[c]{@{}l@{}}$\alpha$\end{tabular} & 
\begin{tabular}[c]{@{}l@{}}$\sigma(\alpha)$\end{tabular} &
\begin{tabular}[c]{@{}l@{}}$...$\end{tabular} &
\begin{tabular}[c]{@{}l@{}}$\sigma^{n-1}(\alpha)$\end{tabular}
\\ \hline
$\mathrm{Id}$ & $\alpha$ & $\sigma(\alpha)$ & $...$ & $\sigma^{n-1}(\alpha)$  \\
$\sigma$ & $\sigma(\alpha)$ & $\sigma^2(\alpha)$ & $...$ & $\alpha$  \\
$...$ & $...$ & $...$ & $...$ & $...$ \\
$\sigma^{n-1}$ & $\sigma^{n-1}(\alpha)$ & $\alpha$ & $...$ & $\sigma^{n-2}(\alpha)$  \\
\end{tabular}
\end{equation*}
and for every $i\in\mathbb{Z}/n\mathbb{Z}$, $\rho(\sigma^i)$ is the circulant matrix whose first column is the 
the $(i+1)$-th vector of the canonical basis of $K^n$.
\end{example}

\begin{example}\label{exeigQ3}
Let $E=\mathbb{Q}_3(\alpha)$, where $\alpha$ is a root of the polynomial $f(x)=x^3+3\in\mathbb{Q}_3[x]$. Then $E/\mathbb{Q}_3$ is a non-Galois extension with normal closure $L=E(z)$, where $z=\sqrt{-3}$. 

The roots of $f$ are $\alpha,\xi\alpha,\xi^2\alpha$, where $\xi=\frac{-1+z}{2}$. Let $r$ be the $\mathbb{Q}_3$-automorphism of $L$ given by $r(\alpha)=\xi\alpha$ and $r(z)=z$, and let $s$ be the one given by $s(z)=-z$ and $s(\alpha)=\alpha$. As permutations of the roots of $f$, $r=(\alpha,\xi\alpha,\xi^2\alpha)$ and $s=(\xi\alpha,\xi^2\alpha)$. These two elements generate $\mathrm{Gal}(L/\mathbb{Q}_3)$, which is dihedral of order $6$. Additionally, $s$ generates $G'=\mathrm{Gal}(L/E)$.

Let $X=G/G'$ and consider  $\lambda\colon G\longrightarrow\mathrm{Perm}(X)$ the morphism given by the action of $G$ on $X$ by left translation. By Greither-Pareigis theorem, since $Perm(X)\cong S_3$ has a unique subgroup of order $3$ and it is normalized by $\lambda(G)$, we know that there is a unique Hopf Galois structure of $E/\mathbb{Q}_3$.  The corresponding  $\mathbb{Q}_3$-Hopf algebra $H_1$ has $\mathbb{Q}_3$-basis $$w_1=\mathrm{Id},\,w_2=(\lambda(r)-\lambda(r)^{-1})z,\,w_3=\lambda(r)+\lambda(r)^{-1}.$$ The action of $H_1$ on $E$ is given by
\begin{equation}\label{table4}
\begin{tabular}{c|ccc} & \begin{tabular}[c]{@{}l@{}}$1$\end{tabular} & 
\begin{tabular}[c]{@{}l@{}}$\alpha$\end{tabular} &
\begin{tabular}[c]{@{}l@{}}$\alpha^2$\end{tabular}
\\ \hline
$w_1$ & $1$ & $\alpha$ & $\alpha^2$ \\
$w_2$ & $0$ & $3\alpha$ &$-3\alpha^2$ \\
$w_3$ & $2$ & $-\alpha$ & $-\alpha^2$ \\
\end{tabular},
\end{equation}
Then, the matrices representing $w_1,w_2$ and $w_3$ are
$$
\begin{pmatrix}
1 & 0 & 0\\
0 & 1 & 0 \\
0 & 0 & 1 \\ 
\end{pmatrix}
\qquad
\begin{pmatrix}
0 & 0 & 0 \\
0 & 3 & 0 \\
0 & 0 & -3 \\
\end{pmatrix}\qquad
\begin{pmatrix}
2 & 0 & 0 \\
0 & -1 & 0 \\
0 & 0 & -1
\end{pmatrix}.$$
On the other hand, the matrices representing multiplication by $1$, $\alpha$ and $\alpha^2$ are
$$
\begin{pmatrix}
1 & 0 & 0\\
0 & 1 & 0 \\
0 & 0 & 1 \\ 
\end{pmatrix}\qquad
\begin{pmatrix} 0&0&-3\\1&0&0
\\ 0&1&0\end{pmatrix}
\qquad 
\begin{pmatrix} 0&-3&0\\ 0&0&-3
\\1&0&0\end{pmatrix}.
$$
By multiplying these two families we can see that we obtain nine linearly independent matrices and explicitly check that $j:E\otimes_{\Q} H_1\to \End_{\Q}(E)$ is an isomorphism.
\end{example}

When we consider the classical Galois action, since elements of $G$ act as automorphisms in $L$, we have a canonical basis of the Hopf algebra whose action is represented with nonsingular matrices. As we have seen in the last example, for a general Hopf Galois action we can have elements in a basis of the Hopf algebra which act through non invertible matrices. 

\section{Associated orders of Hopf Galois actions}\label{sectionintegral}

Let $\mathcal{O}_K$ be a principal ideal domain with field of
fractions $K$. Let $L/K$ be a separable Hopf Galois extension of degree $n$ and let $\mathcal{O}_L$ be the integral closure of
$\mathcal{O}_K$ in $L$. If $H$ is the Hopf algebra of a Hopf Galois structure of $L/K$, then
the associated $\mathcal{O}_K$-order to $\mathcal{O}_L$ in $H$ is $$\mathfrak{A}_H=\{\alpha\in H\,|\,\alpha\cdot x\in\mathcal{O}_L\,\hbox{ for all }x\in\mathcal{O}_L\}.$$

In this section we establish a general method to compute an $\mathcal{O}_K$-basis of $\mathfrak{A}_H$.

\subsection{Motivating example}

In order to compute an $\mathcal{O}_K$-basis of $\mathfrak{A}_H$, we will use the space of solutions of a linear system with coefficient matrix that will be called $M(H,L)$. To motivate the construction let us go back to Example \ref{exeigQ3}. There, since the irreducible polynomial of $\alpha$ is $3$-Eisenstein, we have $\mathcal{O}_E=\mathbb{Z}_3[\alpha]$.  
An element $h_1w_1+h_2w_2+h_3w_3\in H_1$ belongs to the associated order if, and only if, $(h_1w_1+h_2w_2+h_3w_3)(x_1+x_2\alpha+x_3\alpha^2)\in\mathbb{Z}_3[\alpha]$ for all $x_1,x_2,x_3\in \mathbb{Z}_3$. This condition becomes 
$$
\left(h_1\begin{pmatrix}
1 & 0 & 0\\
0 & 1 & 0 \\
0 & 0 & 1 \\ 
\end{pmatrix}
+h_2 
\begin{pmatrix}
0 & 0 & 0 \\
0 & 3 & 0 \\
0 & 0 & -3 \\
\end{pmatrix}+h_3
\begin{pmatrix}
2 & 0 & 0 \\
0 & -1 & 0 \\
0 & 0 & -1
\end{pmatrix}\right)\begin{pmatrix}
x_1\\
x_2\\
x_3
\end{pmatrix}\in \mathbb{Z}_3^3$$
for all $x_1,x_2,x_3\in \mathbb{Z}_3$.
We could write this condition as
$$
(h_1,h_2,h_3)\begin{pmatrix}
w_1\\
w_2\\
w_3
\end{pmatrix}\in\mathrm{Mat}_3(\mathbb{Z}_3)
$$
using matrix representation of the basis elements $w_i$ and considering the second matrix as a block matrix.
The condition we obtain is
$$\begin{pmatrix}
h_1+2h_3 & 0 & 0 \\
0 & h_1+3h_2-h_3 & 0 \\
0 & 0 & h_1-3h_2-h_3
\end{pmatrix}\in \mathrm{Mat}_3(\mathbb{Z}_3)$$ and instead of the previous action on block matrices we are going to consider matrix actions of standard linear algebra:
$$\begin{pmatrix}
1 & 0 & 2 \\
0 & 0 & 0 \\
0 & 0 & 0 \\ 
0 & 0 & 0 \\
1 & 3 & -1 \\
0 & 0 & 0 \\ 
0 & 0 & 0 \\ 
0 & 0 & 0 \\ 
1 & -3 & -1
\end{pmatrix}
\begin{pmatrix}
h_1\\
h_2\\
h_3
\end{pmatrix}\in \mathbb{Z}_3^9
$$
Therefore, in the general situation, once we have fixed a $K$-basis of $L$ 
and a $K$-basis of $H$, we are going to use the $n^2 \times n$ matrix whose $j$th column is formed by concatenating the columns of the matrix for $w_j$.

Let us observe that in this example, after deleting zero rows, this last condition becomes 
$$
\begin{pmatrix}
1 & 0 & 2\\
1 & 3 & -1\\
1 & -3 & -1\\
\end{pmatrix}
\begin{pmatrix}
h_1\\
h_2\\
h_3
\end{pmatrix}\in \mathbb{Z}_3^3
$$
Now, this matrix is invertible with inverse
$$\frac{1}{6}\begin{pmatrix}
2 & 2 & 2 \\
0 & 1 & -1 \\
2 & -1 & -1
\end{pmatrix}.$$ Then, we have that $h=\sum_{i=1}^3h_iw_i\in\mathfrak{A}_H$ if and only if there are some $c_1,c_2,c_3\in\mathbb{Z}_3$ such that 
$$
\begin{pmatrix}
h_1\\
h_2\\
h_3
\end{pmatrix}=
\frac{1}{6}\begin{pmatrix}
2 & 2 & 2 \\
0 & 1 & -1 \\
2 & -1 & -1
\end{pmatrix}
\begin{pmatrix}
c_1\\
c_2\\
c_3
\end{pmatrix},
$$ that is, if and only if \begin{equation*}
    \begin{split}
        h&=\frac{1}{6}(2c_1+2c_2+2c_3)w_1+\frac{1}{6}(c_2-c_3)w_2+\frac{1}{6}(2c_1-c_2-c_3)w_3\\&=\frac{w_1+w_3}{3}c_1+\frac{2w_1+w_2-w_3}{6}c_2+\frac{2w_1-w_2-w_3}{6}c_3
    \end{split}
\end{equation*} for some $c_1,c_2,c_3\in\mathbb{Z}_3$.
We find a basis for the associated order 
$$\mathfrak{A}_{H_1}=\mathbb{Z}_3\left[\frac{w_1+w_3}{3},\frac{2w_1+w_2-w_3}{6},\frac{2w_1-w_2-w_3}{6}\right].$$
The action of this basis on the basis of $\mathcal{O}_E$ is 
\begin{equation}\label{table4bis}
\begin{tabular}{c|ccc} & \begin{tabular}[c]{@{}l@{}}$1$\end{tabular} & 
\begin{tabular}[c]{@{}l@{}}$\alpha$\end{tabular} &
\begin{tabular}[c]{@{}l@{}}$\alpha^2$\end{tabular}
\\ \hline
$b_1=\dfrac{w_1+w_3}{3}$ & $1$ & $0$& $0$ \\[1ex]
$b_2=\dfrac{2w_1+w_2-w_3}{6}$& $0$& $\alpha$ & $0$ \\[1ex]
$b_3=\dfrac{2w_1-w_2-w_3}{6}$ & $0$ & $0$ & $\alpha^2$ \\
\end{tabular}
\end{equation}

This example is related to the non-classical Hopf Galois structure of Chase and Sweedler (\cite[section 4, pp. 35-39]{chasesweedler}).
The ring $\mathbb{Z}_3[\alpha]$ is a $C_3$-graded 
$\mathbb{Z}_3$-algebra and the associated order 
$\mathfrak{A}_{H_1}$ is 
$\mathbb{Z}_3[C_3]^*$, the dual of the group ring, 
which acts on $\mathcal O_E$ as $b_i\cdot \alpha^j=\delta_{ij}\alpha^j$.

\subsection{$n^2\times n$ matrix attached to a Hopf Galois representation}

For $1\le i,j\le n$, let $E_{ij}\in \mathrm{Mat}_n(K)$ be the matrix having $1$ in position $(i,j)$ and $0$ elsewhere. We consider this canonical basis of $\mathrm{Mat}_n(K)$ in the following order: $E_{11},E_{21},\dots E_{n1},E_{12},\dots E_{nn}$. Therefore, we have an isomorphism $\varphi:\mathrm{Mat}_n(K)\to K^{n^2}$
where the image of a matrix becomes the $n^2$ column vector formed by its ordered columns. We have 
$\varphi(E_{ij})=e_{n(j-1)+i}$, vector of the canonical basis of $K^{n^2}$.

\begin{defi}\label{defimatrixaction}  Let $L/K$ be a degree $n$ $H$-Galois field extension.  
Given a $K$-basis $\{\gamma_j\}_{j=1}^{n}$ of $L$ and a $K$-basis $\{w_i\}_{i=1}^{n}$ of $H$, 
as in Section \ref{secmatrixaction} we denote by $w_i\in \mathrm{Mat}_n(K)$ the matrix of $\rho_H(w_i)$ in base $\{\gamma_1,\dots, \gamma_{n}\}$. 
We define the \textbf{matrix of the action of $H$ on $L$} 
as
$$
M(H,L)=\begin{pmatrix}
|& | &\dots  &|  \\
\varphi(w_1)&\varphi(w_2)&\dots &\varphi(w_{n}) \\
|& |&\dots & |\\
\end{pmatrix}\in \mathrm{Mat}_{n^2\times n}(K),
$$
namely, the columns of $M(H,L)$ are obtained transforming into column vectors the matrices of a basis of $H$ acting on a given basis of $L$.
\end{defi}

\begin{pro} Let $L/K$ be a degree $n$ Hopf Galois extension with Hopf algebra $H$. Then, the matrix $M(H,L)$ has rank $n$.
\end{pro}
\begin{proof}
Since $w_1,\dots,w_{n}$ is a $K$-basis of $H$ and $j:L\otimes_K H\to \End_K(L)$ is an isomorphism, the corresponding matrices $w_i$ are $K$-linearly independent and the same happens to the vectors $\varphi(w_i)$, which are the columns of $M(H,L)$.
\end{proof}

This is a convenient way to store all the information of the Hopf action in a single object. If we think of it as a  matrix defined by blocks  
$$M(H,L)=\begin{pmatrix}
\\[-1ex]
M_1(H,L) \\[1ex] 
\hline\\[-1ex]
\cdots \\[1ex]
\hline \\[-1ex]
M_{n}(H,L)\\
\\[-1ex]
\end{pmatrix}$$
then the block $M_j(H,L)$ provides the action of the Hopf algebra on the element $\gamma_j$ of the basis, namely it is the matrix of the $K$-linear map
$$
\begin{array}{ccc}
     H&\to& L  \\
     h&\mapsto& h\cdot \gamma_j
\end{array}
$$
in the chosen bases of $H$ and $L$.

\subsection{Computing a basis of ${\mathfrak A}_H$}

Let us return to the problem of basis computation for the 
associated order $\mathfrak{A}_H$ attached to a Hopf Galois structure of an extension $L/K$ with the hypothesis and notations stated at the beginning of this section, which ensure the existence of an integral basis.

We fix again a $K$-basis $\{w_i\}_{i=1}^{n}$ of $H$, but now we take an $\mathcal{O}_K$-basis $\{\gamma_j\}_{j=1}^{n}$ of $\mathcal{O}_L$. This is, in particular, a $K$-basis of $L$.

\begin{teo}\label{teomatrixassocorder} Let $h=\sum_{i=1}^{n}h_iw_i\in H$, $h_i\in K$. Then, $$h\in\mathfrak{A}_H \iff  M(H,L)\begin{pmatrix}h_1 \\ \vdots \\ h_{n}\end{pmatrix}\in\mathcal{O}_K^{n^2}.$$
\end{teo}
\begin{proof}
By definition, $h\in\mathfrak{A}_H$ if and only if $h\cdot x\in\mathcal{O}_L$ for all $x\in\mathcal{O}_L$. Fix $x\in\mathcal{O}_L$. Since $\{\gamma_j\}_{j=1}^{n}$ is an $\mathcal{O}_K$-basis of $\mathcal{O}_L$, we can write $x=\sum_{j=1}^{n}x_j\gamma_j$, with $x_j\in\mathcal{O}_K$, $1\leq j\leq n$. We compute \begin{equation*}
        h \cdot x=\left(\sum_{i=1}^{n}h_iw_i\right)\cdot\left(\sum_{j=1}^{n}x_j\gamma_j\right)=\sum_{i=1}^{n}\sum_{j=1}^{n}h_ix_jw_i\cdot\gamma_j=
        \sum_{j=1}^{n}x_jM_j(H,L)\begin{pmatrix}
        h_1 \\ \vdots \\ h_{n}\end{pmatrix}.
\end{equation*}
Therefore, $h\in\mathfrak{A}_H$ if and only if
$$M_j(H,L)\begin{pmatrix}
        h_1 \\ \vdots \\ h_{n}\end{pmatrix} 
        \in\mathcal{O}_K^{n}$$
for all $1\le j\le n$, which is equivalent to the condition of the
statement. 
\end{proof} 

In order to characterize elements $h=\sum_{i=1}^{n}h_iw_i\in\mathfrak{A}_H$, we look for an expression for the coordinate vector $(h_1,...,h_{n})$. The previous result says that it becomes a vector of integers when multiplied by $M(H,L)$. If we could replace $M(H,L)$ with an invertible matrix
 we would be able to express such vector as the inverse matrix applied to a vector of integers. To this end we can take for example the Hermite normal form of the integral matrix obtained from $M(H,L)$ removing common denominators.

\begin{teo}[Hermite normal form]\label{hermitenormalform}
Let $A$ be a PID and $M\in M_{m\times n}(A)$ a matrix of rank $n$. Then, there exists an $m\times m$ unimodular matrix $U$, that is $U\in\GL_m(A)$, such that $UM$ is a matrix in Hermite normal form (in particular in row echelon form). 
\end{teo}
\begin{proof}
See \cite[Chapter 5, Theorem 3.1]{adkinsweintraub}.
\end{proof}

Since $K$ is the field of fractions of $\mathcal{O}_K$, we can  write
$M(H,L)=d_M\, M$ with $d_M\in K$ and $M\in  M_{m\times n}(\mathcal{O}_K)$ such that coefficients in $M$ are coprime. Let $U\in\GL_{n^2}(\mathcal{O}_K)$ an unimodular matrix such that $UM$ is in Hermite normal form, and let
$D\in \mathrm{Mat}_n(\mathcal{O}_K)$ the matrix obtained deleting the zero rows of $UM$. This matrix belongs to $\GL_n(K)$. Then, 
\begin{equation*}
    \begin{split}
        h\in\mathfrak{A}_H &\iff  M(H,L)\begin{pmatrix}h_1 \\ \vdots \\ h_{n}\end{pmatrix}\in\mathcal{O}_K^{n^2}
    \iff  d_M\,  M\begin{pmatrix}h_1 \\ \vdots \\ h_{n}\end{pmatrix}\in\mathcal{O}_K^{n^2}\\
    &\iff  d_M\, UM\begin{pmatrix}h_1 \\ \vdots \\ h_{n}\end{pmatrix}\in\mathcal{O}_K^{n^2}\iff d_M D \begin{pmatrix}h_1 \\ \vdots \\ h_{n}\end{pmatrix}\in\mathcal{O}_K^{n}\\
     &\iff  \begin{pmatrix}h_1 \\ \vdots \\ h_{n}\end{pmatrix}
     =\dfrac{1}{d_M} D^{-1}\begin{pmatrix}c_1 \\ \vdots \\ c_{n}\end{pmatrix}
    \end{split}
\end{equation*}
 for some $c_1,\dots,c_{n}$ in $\mathcal{O}_K$.

\begin{teo}\label{basisassocorder} If $\dfrac{1}{d_M} D^{-1}=(d_{ij})_{i,j=1}^{n}$, then the elements 
$$v_i=\sum_{l=1}^{n}d_{li}w_l,\quad 0\leq i\leq n-1$$ form an $\mathcal{O}_K$-basis of $\mathfrak{A}_H$. Namely, the columns of 
$\dfrac{1}{d_M} D^{-1}$ are coordinates of basis elements for $\mathfrak{A}_H$ in the fixed basis of $H$.
The action of $v_i\in\mathfrak{A}_H$ on $\mathcal{O}_L$ is given by
$$
v_i\cdot\gamma_j=M_j(H,L)\begin{pmatrix}d_{1i} \\ \vdots \\ d_{ni}\end{pmatrix}
$$
where resulting coordinates refer to the chosen 
$\mathcal{O}_K$-basis of $\mathcal{O}_L$.
\end{teo}
\begin{proof}
Let $h=\sum_{l=1}^{n}h_l w_l\in H$. By Theorem \ref{hermitenormalform}, 
$$
h\in\mathfrak{A}_H
\iff
\begin{pmatrix}h_1 \\ \vdots \\ h_{n}\end{pmatrix}
     =\dfrac{1}{d_M} D^{-1}\begin{pmatrix}c_1 \\ \vdots \\ c_{n}\end{pmatrix}
 \iff
 h=\sum_{i=1}^{n}c_i\begin{pmatrix}d_{1i} \\ \vdots \\ d_{ni}\end{pmatrix} \iff h\in\langle v_1,...,v_{n}\rangle_{\mathcal{O}_K}.
 $$
Then $\{v_1,...,v_{n}\}$ is a system of generators of $\mathfrak{A}_H$. Now, it is $K$-linearly independent because $\dfrac{1}{d_M} D^{-1}\in \GL_n(K)$ and $\{w_i\}_{i=1}^{n}$ is a $K$-basis, so it is also $\mathcal{O}_K$-linearly independent and hence an $\mathcal{O}_K$-basis for $\mathfrak{A}_H$.

As we said before, $M_j(H,L)$ is the matrix of the $K$-linear map
$$
\begin{array}{ccc}
     H&\to& L  \\
     h&\mapsto& h\cdot \gamma_j
\end{array}
$$
in the chosen basis. Therefore, in order to obtain 
$v_i\cdot\gamma_j$ we just need the coordinates of $v_i$ in basis
$w_1,\dots,w_n$ and these are precisely the elements of the $i$-th column of $\dfrac{1}{d_M} D^{-1}$.
\end{proof} 

\begin{rmk}
We haven't made use of the upper triangular nature of $\dfrac{1}{d_M} D^{-1}$ because the result is independent from that. For any $D_U\in \mathrm{Mat}_n(K)$ such that 
$$
U\, M(H,L)=\begin{pmatrix}
D_U\\ \hline \\[-2ex] O
\end{pmatrix}
$$
with $U\in GL_{n^2}(\mathcal{O}_K)$, we will have $D_U\in \GL_n(K)$ and the columns of $D_U^{-1}$ will provide an $\mathcal{O}_K$-basis for the associated order. 
\end{rmk}

\begin{example}\label{quadratic}
Let $L=K(\sqrt a\,)$ be a quadratic Galois extension.
Then, its unique Hopf Galois structure is the classical one, which has $H=K[1,\sigma]$ with $\sigma(\sqrt a\, )=-\sqrt a$.  Consider the basis $\{1,\sigma\}$ of $H$ and assume that $\{1,\sqrt{a}\}$ is an integral basis, then the matrix of the action is
$$M(H,L)=\begin{pmatrix}
1 & 1 \\
0 & 0 \\ 
0 & 0 \\
1 & -1 \\
\end{pmatrix}.
$$
With a permutation of rows, we get 
$$D_U=
\begin{pmatrix}
 1&1\\ 1&-1
 \end{pmatrix}.
$$
and the columns of the inverse matrix provide the basis 
of the associated order (which is the maximal order) of orthogonal idempotents
$$
\dfrac{1+\sigma}{2}, \dfrac{1-\sigma}{2}.
$$
On the other hand, computation of the Hermite normal form gives 
$$
D=
\begin{pmatrix}
 1&1\\ 0&2
 \end{pmatrix}.
$$
which provides the basis $1,\dfrac{-1+\sigma}{2}$. If we transform $D$ into  
$$
\begin{pmatrix}
 1&-1\\ 0&2
 \end{pmatrix}.
$$
we obtain basis $1,\dfrac{1+\sigma}{2}$.
\end{example}

\begin{example}\label{exeigQ3m}
In the above motivating example, we took a matrix $U$ which was just a permutation matrix acting on the rows of $M(H,L)$ and compute a basis using the corresponding $D_U$.  We have $d_M=1$ and when we compute the Hermite normal form of $M(H,L)$ we obtain 
$$
D=\begin{pmatrix}
1&0&2\\0&3&3\\0&0&6
\end{pmatrix}\,,
\qquad D^{-1}=
\begin{pmatrix}
 1&0&-1/3\\ 0&1/3&-1/6
\\ 0&0&1/6
\end{pmatrix}.
$$
Therefore, another basis for the associated order would be
$$
w_1,\  \dfrac13 w_2, \  -\dfrac16(2w_1+w_2-w_3).
$$
If we perform further reduction steps on matrix $D$
$$
\begin{pmatrix}
1&0&2\\
0&3&3\\
0&0&6
\end{pmatrix}\to 
\begin{pmatrix}
1&0&2\\
0&3&3\\
0&0&3
\end{pmatrix}\to 
\begin{pmatrix}
1&0&-1\\
0&3&0\\
0&0&3
\end{pmatrix}
$$
we obtain yet another basis: 
$$
w_1,\  \dfrac{w_2}3, \  \dfrac{w_1+w_3}3.
$$
\end{example}

\begin{example}\label{exeigQ3'}
Now we consider $E=\mathbb{Q}_3(\alpha)$ with $\alpha$ a root of $f(x)=x^3+3x^2+3\in\mathbb{Q}_3[x]$. As in Example \ref{exeigQ3}, the normal closure $L/\mathbb{Q}_3$ is dihedral of degree $6$, and the cubic extension $E/\mathbb{Q}_3$ has an unique Hopf Galois structure with Hopf algebra $H_1$ generated by elements $w_1$, $w_2$ and $w_3$ defined in a similar way; namely, we take a $3$-cycle $r$ of $\mathrm{Gal}(L/\mathbb{Q}_3)$ and $z=\sqrt{-1}$. With respect to this $\mathbb{Q}_3$-basis of $H_1$ and the $\mathbb{Q}_3$-basis of $L$ corresponding to the powers of $\alpha$, the action of $H_1$ on $E$ is given by
\begin{equation}\label{table5}
\begin{tabular}{c|ccc} & \begin{tabular}[c]{@{}l@{}}$1$\end{tabular} & 
\begin{tabular}[c]{@{}l@{}}$\alpha$\end{tabular} &
\begin{tabular}[c]{@{}l@{}}$\alpha^2$\end{tabular}
\\ \hline
$w_1$ & $1$ & $\alpha$ & $\alpha^2$ \\
$w_2$ & $0$ & $3+9\alpha+2\alpha^2$ & $-3-30\alpha-9\alpha^2$ \\
$w_3$ & $2$ & $-3-\alpha$ &$9-\alpha^2$ \\
\end{tabular},
\end{equation}
Since $\alpha$ is a root of a $3$-Eisenstein polynomial, powers of $\alpha$ give a $\mathbb{Z}_3$-basis of $\mathcal{O}_E$.  We can apply the previous method to compute the associated order $\mathfrak{A}_{H_1}$. We have
$$
\begin{pmatrix}1&0&0&0&0&0&0&0&0
\\0&0&0&-1&0&2&0&0&0\\ 0&0&0&-2&0
&3&0&0&0\\0&1&0&0&0&0&0&0&0\\0&0
&1&0&0&0&0&0&0\\ -1&0&0&-1&1&-3&0&0&0
\\ 0&0&0&3&0&-3&1&0&0\\
0&0&0&0&0&
15&0&1&0\\ 
-1&0&0&-1&0&6&0&0&1\end {pmatrix}
\begin{pmatrix}
1 & 0 & 2 \\
0 & 0 & 0 \\
0 & 0 & 0 \\ 
0 & 3 & -3 \\
1 & 9 & -1 \\
0 & 2 & 0 \\ 
0 & -3 & 9 \\ 
0 & -30 & 0 \\ 
1 & -9 & -1
\end{pmatrix}=
\begin{pmatrix}
 1&0&2\\ 
 0&1&3\\ 
 0&0&6\\ 
 0&0&0\\
 0&0&0\\
  0&0&0\\
 0&0&0\\
  0&0&0\\
 0&0&0\\
\end{pmatrix}.$$ 
Hence,
$$\mathfrak{A}_{H_1}=\mathbb{Z}_3\left[w_1,w_2,
\frac{-2w_1-3w_2+w_3}{6}\right].$$
If we perform further reduction steps on matrix $D$
$$
\begin{pmatrix}
1&0&2\\
0&1&3\\
0&0&6
\end{pmatrix}\to 
\begin{pmatrix}
1&0&2\\
0&1&3\\
0&0&3
\end{pmatrix}\to 
\begin{pmatrix}
1&0&-1\\
0&1&0\\
0&0&3
\end{pmatrix}
$$
we obtain  
$$\mathfrak{A}_{H_1}=\mathbb{Z}_3\left[w_1,w_2,
\dfrac{w_1+w_3}{3}\right].$$
\end{example}

\section{Freeness of the ring of integers over the associated order}\label{secfreeness}

We keep the hypothesis of Section \ref{sectionintegral}. In this part we deal with the problem of the structure of $\mathcal{O}_L$ as $\mathfrak{A}_H$-module. Namely, we would like to find conditions for the freeness of such module.

\subsection{Characterization of freeness}

Let $H$ be the Hopf algebra acting on $L$ and let $\{w_i\}_{i=1}^{n}$ be a $K$-basis of $H$ and $\{\gamma_j\}_{j=1}^{n}$ an $\mathcal{O}_K$-basis of $\mathcal{O}_L$. We proved in Theorem \ref{basisassocorder} that the elements $$v_i=\sum_{l=1}^{n}d_{li}w_l,\,1\leq i\leq n$$ form an $\mathcal{O}_K$-basis of $\mathfrak{A}_H$. Standard linear algebra yields that an element $\beta\in\mathcal{O}_L$ is a free generator of $\mathcal{O}_L$ as $\mathfrak{A}_H$-module if and only if the elements $v_i\cdot\beta\in\mathcal{O}_L$ form an $\mathcal{O}_K$-basis of $\mathcal{O}_L$, that is, such set is both $\mathcal{O}_K$-linearly independent and an $\mathcal{O}_K$-system of generators of $\mathcal{O}_L$.

For a given element $\beta\in\mathcal{O}_L$, we denote by $\mathcal{D}_{\beta}(H,L)$ the matrix whose $i$-th column are the coordinates of $v_i\cdot\beta$ with respect to the fixed $\mathcal{O}_K$-basis of $\mathcal{O}_L$. That is, if 
$\beta=\sum_{j=1}^{n}\beta_j\gamma_j$ and $D^{-1}$ is the matrix whose 
columns provide the basis of the associated order, then
$$
\mathcal{D}_{\beta}(H,L)=\sum_{j=1}^{n}\beta_jM_j(H,L)D^{-1}.
$$

It is trivial that  $\det{\mathcal{D}_{\beta}(H,L)}\neq0$ is a necessary condition for the elements $v_i\cdot\beta$ being an $\mathfrak{A}_H$-basis of $\mathcal{O}_L$. However, it does not imply that $\{v_i\cdot\beta\}_{i=1}^{n}$ is a system of generators. What is true is that it is equivalent to the unimodularity of $\mathcal{D}_{\beta}(H,L)$. In order to prove this, we assume that $\mathcal{D}_{\beta}(H,L)$ is unimodular and we recall the following.

\begin{pro} Let $R$ be a noetherian integral domain with field of fractions $K$ and let $A\subset B$ be $R$-algebras which are $R$-orders in a finite separable $K$-algebra $L$. Then $A=B$ if and only if $\mathrm{disc}_R(A)=\mathrm{disc}_R(B)$.
\end{pro}
\begin{proof}
See \cite[(22.4)]{childs}.
\end{proof}

In our case,  $\mathfrak{A}_H\cdot\beta\subset\mathcal{O}_L$, and whenever the $v_i\cdot\beta$ are $\mathcal{O}_K$-linearly independent $\mathfrak{A}_H\cdot\beta$ is an $\mathcal{O}_K$-order of $L$. This is the case assuming $\mathcal{D}_{\beta}(H,L)\in\GL_n(\mathcal{O}_K)$. By \cite[(22.2)]{childs}, $$\mathrm{disc}_{\mathcal{O}_K}(\mathfrak{A}_H\cdot\beta)=[\mathcal{O}_L:\mathfrak{A}_H\cdot\beta]^2\, \mathrm{disc}_{\mathcal{O}_K}(\mathcal{O}_L),$$ where $[\mathcal{O}_L:\mathfrak{A}_H\cdot\beta]$ is the generalized module index of the free $\mathcal{O}_K$-modules $\mathcal{O}_L$ and $\mathfrak{A}_H\cdot\beta$ as defined in \cite[Page 10]{cassels-frohlich}. Then, the two discriminants coincide if and only if the index $[\mathcal{O}_L:\mathfrak{A}_H\cdot\beta]$ is in $\mathcal{O}_K^*$. We compute the index by taking the $\mathcal{O}_K$-bases $\{v_i\cdot\beta\}_{i=1}^{n}$ of $\mathfrak{A}_H\cdot\beta$ and $\{\gamma_j\}_{j=1}^{n}$ of $\mathcal{O}_L$, and we find that $[\mathcal{O}_L:\mathfrak{A}_H\cdot\beta]=\mathrm{det}(\mathcal{D}_{\beta}(H,L))$, which belongs to $\mathcal{O}_K^*$. Conversely, if $\{v_i\cdot\beta\}_{i=0}^{n-1}$ is an $\mathcal{O}_K$-basis of $\mathcal{O}_L$, then $\mathcal{D}_{\beta}(H,L)$ is a change of basis matrix, so it is unimodular. Hence, we have the following criterion.

\begin{pro} An element $\beta\in\mathcal{O}_L$ is an $\mathfrak{A}_H$-free generator of $\mathcal{O}_L$ if and only if the associated matrix $\mathcal{D}_{\beta}(H,L)$ is unimodular.
\end{pro}

\begin{example}\label{quadraticfreeness} Let $L/K$ be a quadratic extension of $p$-adic fields, which is known to have the classical Galois structure $K[G]$ as its unique Hopf Galois structure. When $p\geq 3$, $L/K$ is tamely ramified and then $\mathcal{O}_L$ is $\mathfrak{A}_{K[G]}$-free. Let $z\in L$ such that $z\notin K$, $z^2\in K$ and $\mathcal{O}_L=\mathcal{O}_K[z]$. 

By Example \ref{quadratic}, a basis of the associated order is $\eta_1=1$, $\eta_2'=\dfrac{-1+\sigma}{2}$. The action of this basis over $L$ is 
\begin{center}
\begin{tabular}{c|cc}
$\mathfrak{A}_{K[G]}$ &$1$&$z$\\ \hline
$\eta_1$ & $1$ & $z$ \\
$\eta_2'$ & $0$ & $-z$ \\
\end{tabular}
\end{center}
Thus, $$\mathcal{D}_{1+z}(K[G],L)=\begin{pmatrix}
1 & 0 \\
1 & -1 
\end{pmatrix}$$ which is in $\GL_n(\mathcal{O}_K)$. Therefore $1+z$ generates $\mathcal{O}_L$ as $\mathfrak{A}_{K[G]}$-module.
\end{example}

\section{Induced Hopf Galois structures}\label{sectioninduced}

Let $L/K$ be a finite Galois extension of fields with Galois group $G$. The Greither-Pareigis theorem applied to this situation gives that Hopf Galois structures of $L/K$ have Hopf algebras of the form $H=L[N]^G,$ where $N$ runs through the regular subgroups of $\mathrm{Perm}(G)$ normalized by $\lambda(G)$. Here,
$\lambda\colon G\to \Perm(G)$ is the left regular representation of $G$.

Assume that $G$ decomposes as a semidirect product $G=J\rtimes G',$ where $J$ is a normal subgroup of $G$, normal complement for the subgroup $G'$. Let us denote $E=L^{G'}$, $F=L^J$, $r=[E:K]$, $u=[F:K]$. We have a lattice of intermediate fields 
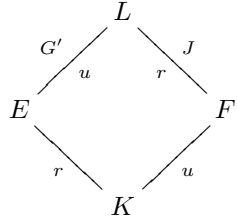
\[\xymatrix{
& \ar@{-}[dl]^u_{G'} L \ar@{-}[dr]^J_r & \\
E \ar@{-}[dr]_r & & \ar@{-}[dl]^u F \\
& K &}
\]
with $E\cap F=K$ and $L=EF$. By \cite[Chapter 5, Theorem 5.5]{cohn}, $E/K$ and $F/K$ are $K$-linearly disjoint. Moreover, since $L/K$ is Galois, $L/E$ is Galois with Galois group $G'$. By the Greither-Pareigis theorem, Hopf Galois structures of $L/E$ are in one-to-one correspondence with regular subgroups of $\mathrm{Perm}(G')$ normalized by $\lambda'(G')$. Here, 
$\lambda':G'\to \mathrm{Perm}(G')$ is the
left regular representation of $G'$. 

On the other hand, if $\widetilde{E}/K$ is the normal closure of $E/K$ in $L/K$, then Hopf Galois structures of $E/K$ are in one-to-one correspondence with regular subgroups of $\mathrm{Perm}(\widetilde{X})$ normalized by $\widetilde{\lambda}(G)$, where $\widetilde{X}=\mathrm{Gal}(\widetilde{E}/K)/\mathrm{Gal}(\widetilde{E}/E)$ and $\tilde{\lambda}:\mathrm{Gal}(\widetilde{E}/K)\to\mathrm{Perm}(\widetilde{X})$ is given by the action of $\mathrm{Gal}(\widetilde{E}/K)$ by left translation on left cosets. 
But the Fundamental Theorem of Galois Theory gives us the isomorphisms 
$$
\mathrm{Gal}(\widetilde{E}/K)\cong\mathrm{Gal}(L/K)/\mathrm{Gal}(L/\widetilde{E}),$$ $$\mathrm{Gal}(\widetilde{E}/E)\cong\mathrm{Gal}(L/E)/\mathrm{Gal}(L/\widetilde{E}).
$$ 
Hence, there is a bijection $\widetilde{X}\cong\mathrm{Gal}(L/K)/\mathrm{Gal}(L/E)=G/G'=X\cong J$.
Taking the actions into account, the action of $\mathrm{Gal}(\widetilde{E}/K)$ by left translation on left cosets corresponds to the action of $G$ on left cosets. As for the second isomorphism, if we let $\sigma_1,\dots,\sigma_r$ of $J$ be a left transversal of $G'$, then the action of an element $g=\sigma\tau\in G$, with $\sigma\in J$ and $\tau\in G'$, is the following: 
$$\sigma\tau\cdot \sigma_i G'=\sigma\tau \sigma_i\tau^{-1}\tau G'=\sigma(\tau\sigma_i\tau^{-1}) G'$$
This gives an action $\lambda_c:G\to \mathrm{Perm}(J)$.

We build a Hopf Galois structure of $L/K$ from Hopf Galois structures of $L/E$ and $E/K$ as follows:

\begin{teo}[Induction]\label{teoinduced}
Assume that $N_1\subseteq\mathrm{Perm}(X)=\mathrm{Perm}(J)$ is regular and normalized by $\lambda_c(G)$, namely it gives $E/K$ a Hopf Galois structure, and $N_2\subseteq\mathrm{Perm}(G')$ is regular and normalized by $\lambda'(G')$, namely it gives $L/E$ a Hopf Galois structure. Then, $L/K$ has a Hopf Galois structure of type $N_1\times N_2$.
\end{teo}
\begin{proof}
See \cite[Theorem 3]{cresporiovela}.
\end{proof}

Hopf Galois structures of $L/K$ are given by the regular subgroups of $\mathrm{Perm}(G)$ normalized by $\lambda(G)$. The definition of such a group relies on the following factorization of $\lambda$ 
$$
\begin{array}{ccccc}
    G &\to &\mathrm{Perm}(J)\times \mathrm{Perm}(G')&\hookrightarrow& \mathrm{Perm}(G)  \\
 \sigma\tau&\mapsto& \big(\lambda_c(\sigma\tau),\lambda'(\tau)\big)&&\\
&& (\varphi,\psi) & \mapsto & \iota(\varphi,\psi)
\end{array}
$$
where $\iota(\varphi,\psi)(\sigma\tau)=
\varphi(\sigma)\psi(\tau).$ The group $N=\iota(N_1\times N_2)$ provides a Hopf Galois structure for $L/K$.

\begin{rmk}
Let us remark that this notion of induction includes the case when $G$ is a direct product. By iteration we could deal for example with families like nilpotent groups, which are direct product of their Sylow subgroups. In this sense, Theorem 5 of \cite{byottNil} would say that for nilpotent Galois extensions all Hopf Galois structures of nilpotent type are induced. 
\end{rmk}

Although a regular subgroup of $\mathrm{Perm}(G)$ normalized by $\lambda(G)$ is enough to completely determine a Hopf Galois structure of $L/K$, our goal is to give a more precise description of the Hopf algebra and the Hopf action of an induced Hopf Galois structure for a better understanding of how induction works and how we can apply the methods of previous section to study associated orders and freeness of the ring of integers. 

\subsection{Induced Hopf algebras}\label{secinducedhopfalg}

Our first aim is to describe the Hopf algebra of an induced Hopf Galois structure of $L/K$ in terms of the Hopf algebras of the Hopf Galois structures from which it is built. To this end, first of all we prove that 
in fact induction can also be defined in terms of Hopf Galois structures of the subextensions with base field $K$.

\begin{pro}\label{proptensorcorresp} Let $L/K$ be a Galois extension with Galois group $G=J\rtimes G'$ and let $E=L^{G'}$, $F=L^J$. Then, there is an one-to-one correspondence between 
Hopf Galois structures of $F/K$ and 
Hopf Galois structures of $L/E$.
\end{pro}
\begin{proof}
Since $J$ is a normal subgroup of $G$, the extension $F/K$ is Galois with Galois group $G/J$. Its Hopf Galois structures are in one to one correspondence  with regular subgroups of $\Perm(G/J)$ normalized by 
$\lambda^{G/J}(G/J)$, where $\lambda^{G/J}:G/J\to \Perm(G/J)$ is the left regular representation of this quotient group.
Via the isomorphism
$$
\begin{array}{c}
G/J \longleftrightarrow G'\\
\overline \tau\longleftrightarrow \tau
\end{array}
$$ 
we can identify $\Perm(G/J)$ with $\Perm(G')$ and under this identification  $\lambda^{G/J}(G/J)$ identifies with $\lambda'(G')$.
Therefore regular subgroups of $\Perm(G')$ normalized by $\lambda'(G')$ are in one-to-one correspondence with regular subgroups of 
$\Perm(G/J)$ normalized by $\lambda^{G/J}(G/J)$.
\end{proof}

If we use the same name $N_2$ for such a regular subgroup the corresponding
Hopf algebras are
$\overline{H}=F[N_2]^{G/J}$ and $H_2=L[N_2]^{G'}$, 
so that $H_2=E\otimes_K \overline H$ and $\overline H=H_2^J$ (see 
\cite[Section 3]{KKTU} for a more general framework).

\begin{rmk} The actions of the Hopf algebras involved work as follows:
\begin{itemize}
    \item $E\otimes_K\overline{H}$ acts on $L=E\otimes_K F$  through the product on $E$ in the first term and the Hopf  action in the second one.
    \item The action of the Hopf Galois structure $H_2=L[N_2]^{G'}$ of $L/E$ on $L$ is $J$-equivariant, namely $\sigma(h\cdot x)=\sigma(h)\cdot\sigma(x)$ 
    for $\sigma\in J$, $h\in H_2$ and $x\in L$. Indeed, $J$ acts on $L$ by the classical Galois action and on $N_2$ by conjugation, but this last action turns out to be trivial (see the proof of \cite[Theorem 3.1]{KKTU}). Consequently, the restricted action of $H_2^J$ on $F=L^J$ makes sense, and this is the Hopf Galois action of $\overline{H}$ on $F$.
\end{itemize}
\end{rmk}
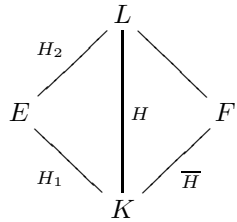
\[\xymatrix{
& \ar@{-}[dl]_{H_2} L \ar@{-}[dr] \ar@{-}[dd]^H & \\
E \ar@{-}[dr]_{H_1} & & \ar@{-}[dl]^{\overline{H}} F \\
& K &}
\] 
According to \cite[Theorem 3.1]{KKTU}, the Hopf algebra of an induced Hopf Galois structure fits in a short exact sequence with the Hopf Galois structures $H_1$ and $\overline{H}$ from which it is induced. But in our case there is a deeper relation: $H$ is actually the tensor product $H_1\otimes_K\overline{H}$.

\begin{pro}\label{propinduced} 
Let $N_1\subseteq\mathrm{Perm}(J)$ be regular and normalized by $\lambda_c(G)$  and let $N_2\subseteq\mathrm{Perm}(G/J)$ be regular and normalized by $\lambda^{G/J}(G/J)$. Then 
$N=\iota(N_1\times N_2)\subseteq\mathrm{Perm}(G)$ gives the induced Hopf Galois structure of $L/K$. Therefore, the corresponding Hopf algebras are $H=L[N]^G$,  
 $H_1=L[N_1]^G$ and $\overline{H}=F[N_2]^{G/J}$. Then, $$H= H_1\otimes_K\overline{H}.$$
\end{pro}
\begin{proof}

$$\begin{array}{rcl}
H&=&L[N]^G=L[\iota(N_1\times N_2)]^G=
\bigl(L[\iota(N_1\times1)\times \iota(1\times N_2)]\bigr)^{G}\\[2ex]
&=&\bigl(L[\iota(N_1\times 1)]\otimes_K L[\iota(1\times N_2)]\bigr)^{G}
=\bigl(L[N_1]\otimes_K L[ N_2]\bigr)^{G}\end{array}
$$
Since the action of conjugation by $\lambda(G)$ on $\Perm(G)$ factors through $\Perm(J)\times \Perm(G/J)$ as conjugation by $\lambda_c(G)$ on 
the first component and conjugation by $\lambda^{G/J}(G/J)$ on the 
second one, we have
$$\begin{array}{rcl}
H&=&L[N_1]^G\otimes_K F[ N_2]^{G/J}=H_1\otimes_K \overline H\end{array}
$$
\end{proof}

\begin{example}\label{exampledihedral} Let us assume that $G\cong D_{2p}$, the dihedral group of order $2p$ with $p$ an odd prime, and let us fix a presentation $G=\langle r,s\mid r^p=s^2=1,\,sr=r^{-1}s\rangle.$ Then $G=J\rtimes G'$ with $J=\langle r\rangle$ the unique order $p$ subgroup of $G$ and $G'$ any of the $p$ different order $2$ subgroups $G'_d=\langle r^ds\rangle$ with $0\leq d\leq p-1$. Therefore, $F=L^J/K$ is the unique degree $2$ subextension of $L/K$, while there are $p$ possible degree $p$ subextensions $E_d/K$ of $L/K$, $E_d=L^{\langle r^ds\rangle}$ (see \cite[Section 4]{cresporiovela2} for further details). Hence, $L/K$ has $p$ induced Hopf Galois structures, which are of cyclic type $C_{2p}$ and have respective Hopf algebras
$$H_d=L[\langle \lambda(r)\rangle]^G\otimes_K F[\langle \lambda'(r^ds)\rangle]^{G/J},\ 0\leq d\leq p-1.$$
The action on the second factor is obtained taking the elements of $G'_d$ as a system of representatives for $G/J$ and make them act by conjugation. Therefore, the action is trivial and we have the group algebra $K[G/J]$. In fact, the classical Galois structure is the unique Hopf Galois structure of a quadratic extension.
\end{example}

\begin{example}
Let $F$ be an imaginary quadratic field, and let $\mathcal{O}$ be the order of the conductor $f$ in $F$. Let $L$ be the ring class field associated to $\mathcal{O}$. Then $L/F$ is an abelian extension and $L/\mathbb{Q}$ is Galois. By \cite[Lemma 9.3]{cox}, $\mathrm{Gal}(L/\mathbb{Q})=\mathrm{Gal}(L/F)\rtimes G',$ where $G'$ is a group of order $2$. Let $E=L^{G'}$. Then, $J=\mathrm{Gal}(L/K)$ is a normal complement of $G'$, so the Hopf algebras of the induced Hopf Galois structures of $L/\mathbb{Q}$ are $H=H_1\otimes_K\overline{H},$ where $H_1$ gives the Hopf Galois structure of $E/\mathbb{Q}$ and $\overline{H}$ is the group algebra of $\mathrm{Gal}(F/\mathbb{Q})$, since the classical one is again the unique Hopf Galois structure of this extension.
\end{example}

\subsection{Induced Hopf actions and their representations}

In this section we study the Hopf action of an induced Hopf Galois structure $H=L[N]^G$ of $L/K$ in terms of the action of $H_1=L[N_1]^G$ on $E$ and the action of $\overline{H}=F[N_2]^{G/J}$ on $F$. 
We prove item $5$ of Theorem \ref{firstmaintheorem} in the following result. The terminology for representations follows \cite[section 3.2]{Serre} which covers the case of $G$ being a direct product and 
considering classical Galois actions.

\begin{pro}\label{actioninduced} 
With the previous notations for induced Hopf Galois structures, let 
$\rho_H:H\to \End_K(L)$, $\rho_{H_1}:H_1\to \End_K(E)$ and 
$\rho_{\overline H}:\overline H\to \End_K(F)$ be the representations
obtained from the respective Hopf actions. 
We have $L=E\otimes_K F$, $H=H_1\otimes \overline{H}$ and 
$$
\rho_H=\rho_{H_1}\otimes\rho_{\overline H}
$$
That is, for $w\in H_1$, $\eta\in\overline{H}$, 
$\alpha\in E$ and $z\in F$,
$(w\otimes\eta)\cdot(\alpha\otimes z)=(w\cdot\alpha)\otimes(\eta\cdot z).$\end{pro}
\begin{proof}
As $w\in L[N_1]^G$ and $\eta\in F[N_2]^{G/J}$, let us write $$w=\sum_{i=1}^rc_in_i^{(1)},\,c_i\in L, \qquad \eta=\sum_{j=1}^ud_jn_j^{(2)},\,d_j\in F,$$ where $N_1=\{n_i^{(1)}\}_{i=1}^r$ and $N_2=\{n_j^{(2)}\}_{j=1}^u$. 
Recall that $\iota(n_i^{(1)},n_j^{(2)})(Id_G)=n_i^{(1)}(\mathrm{Id}_J)n_j^{(2)}(\mathrm{Id}_{G'})$. 
Then:

\begin{equation*}
    \begin{split}
         (w\otimes\eta)\cdot(\alpha\otimes  z)&=\left(\sum_{i=1}^r\sum_{j=1}^uc_id_j\iota\bigl((n_i^{(1)},n_j^{(2)})\bigr)\right)\cdot(\alpha\otimes z)\\&=\sum_{i=1}^r\sum_{j=1}^uc_id_j
         \iota\bigl((n_i^{(1)},n_j^{(2)})\bigr)^{-1}(\mathrm{Id_G})(\alpha z)\\&=\sum_{i=1}^r\sum_{j=1}^uc_id_j
         \iota\bigl(((n_i^{(1)})^{-1},(n_j^{(2)})^{-1})\bigr)(\mathrm{Id_G})(\alpha z)
         \\&=\sum_{i=1}^r\sum_{j=1}^uc_id_j(n_i^{(1)})^{-1}(\mathrm{Id_J})(\alpha)(n_j^{(2)})^{-1}(\mathrm{Id_{G'}})(z)\\&=\left(\sum_{i=1}^rc_i(n_i^{(1)})^{-1}(\mathrm{Id_J})(\alpha)\right)\left(\sum_{j=1}^ud_j(n_j^{(2)})^{-1}(\mathrm{Id_{G'}})(z)\right)\\&=\left(\sum_{i=1}^rc_in_i^{(1)}\right)\cdot\alpha\left(\sum_{j=1}^ud_jn_j^{(2)}\right)\cdot z=(w\cdot\alpha)\otimes(\eta\cdot z).
    \end{split}
\end{equation*}
\end{proof} 
Therefore, by suitable choices of basis, the matrices of the action of $H$ are Kronecker products of the matrices of the actions of $H_1$ and $\overline{H}$.

\begin{example}\label{degree6}
We consider again Example \ref{exeigQ3}. We have
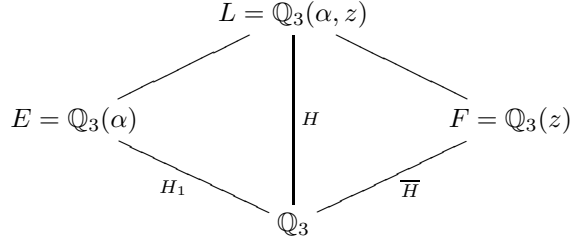
\[\xymatrix{
& \ar@{-}[dl] L=\Q(\alpha,z) \ar@{-}[dr] \ar@{-}[dd]^H & \\
E=\Q(\alpha)\ar@{-}[dr]_{H_1} & & \ar@{-}[dl]^{\overline{H}} F=\Q(z) \\
& \Q &}
\] 
and the action of $H_1$ on $E$ and the action of $\overline{H}$ on $F$ are given by
\begin{equation}\label{table4}
\begin{tabular}{c|ccc} & \begin{tabular}[c]{@{}l@{}}$1$\end{tabular} & 
\begin{tabular}[c]{@{}l@{}}$\alpha$\end{tabular} &
\begin{tabular}[c]{@{}l@{}}$\alpha^2$\end{tabular}
\\ \hline
$w_1$ & $1$ & $\alpha$ & $\alpha^2$ \\
$w_2$ & $0$ & $3\alpha$ &$-3\alpha^2$ \\
$w_3$ & $2$ & $-\alpha$ & $-\alpha^2$ \\
\end{tabular}
\qquad\quad
\begin{tabular}{ c |c c  }
& 1 & $z$ \\
\hline 
$\eta_1$ & 1 & $z$  \\
$\eta_2$ & 1 & $-z$ \\
\end{tabular},
\end{equation}
respectively. The  Hopf action of the induced structure is immediately computed from these two actions. One can also check that the structure is given by $N=\langle g\rangle\subseteq\Perm(G)$, where $g=(1,rs,r^2,s,r,r^2s)$, and the Hopf action goes as follows:
\begin{center}
\begin{tabular}{ r |rrrrrr}
& $1$ &$z$& $\alpha$  & $\alpha z$ &$\alpha^2$ &$\alpha^2 z$ \\
\hline 
$id=w_1\otimes \eta_1$ & 
$1$&$z$&$\alpha$&$\alpha z$&$\alpha^2$&$\alpha^2 z$ \\
$g^3=w_1\otimes \eta_2$ 
& $1$ &$-z$& $\alpha$ & $-\alpha z$ &$\alpha^2$   &$-\alpha^2 z$\\
$z(g^{-2}-g^2)=w_2\otimes \eta_1$  
& $0$ & $0$& $3\alpha$  & $3\alpha z$ &$-3\alpha^2$      &$-3\alpha^2 z$ \\
$z(g-g^{-1})=w_2\otimes \eta_2$ 
&$0$& $0$ & $3\alpha$ & $-3\alpha z$ &$-3\alpha^2$      &$3\alpha^2 z$\\
$g^2+g^{-2}=w_3\otimes \eta_1$ 
& $2$& $2z$ & $-\alpha$    & $-\alpha z$&$-\alpha^2$      
 &$-\alpha^2 z$ \\
$g+g^{-1}=w_3\otimes \eta_2$ 
& $2$ & $-2z$ & $-\alpha$ & $\alpha z$   &$-\alpha^2$    &$\alpha^2 z$\\
\end{tabular}
\end{center}
By fixing the basis $\{1,z,\alpha,\alpha z,\alpha^2,\alpha^2z\}$ of $L$, we can compute the matrix representing the element $w_i\otimes\eta_j$ from the previous table, which coincides with the Kronecker product of the ones representing $w_i$ and $\eta_j$. For example, the matrix representing $w_2\otimes \eta_2$ is
$$
\left(\begin{array}{cccccc}
0 & 0 & 0 &0 & 0 & 0\\
0 & 0 & 0 & 0 & 0 & 0 \\
0 & 0 & 3 &0 & 0 & 0\\
0 & 0 & 0 &-3 & 0 & 0\\
0 & 0 & 0 &0 & -3 & 0 \\
0 & 0 & 0 & 0 & 0 & 3\\
\end{array}\right)= \begin{pmatrix}
0 & 0 & 0 \\
0 & 3 & 0 \\
0 & 0 & -3 \\
\end{pmatrix} \otimes\begin{pmatrix}
1 & 0 \\
0 & -1
\end{pmatrix}.
$$

\end{example}

\subsection{Associated order of an induced Hopf Galois action}

Let us assume that $L/K$ is any field extension with an induced 
Hopf Galois structure. 
We recall the setting of Section \ref{secinducedhopfalg}: $G=J\rtimes G'$
$E=L^{G'}$ and $F=L^J$, $E/K$ is $H_1$-Galois and $F/K$ is $\overline{H}$-Galois and we have  $H=H_1\otimes_K\overline{H}$.

We assume that the Hopf Galois structure is induced and has $H=H_1\otimes_K\overline{H}$. First of all we analyze the behaviour of the matrix of the action with respect to the tensor product. 
Let us fix a $K$-basis $\{w_i\}_{i=1}^{r}$ of $H_1$ and a $K$-basis $\{\eta_j\}_{j=1}^{u}$ of $\overline{H}$.
We also fix $K$-bases 
$\{\alpha_k\}_{k=1}^{r}$ of $E$ and  $\{z_l\}_{l=1}^{u}$  of $F$

\begin{teo}\label{actioninducedmatrices} 
We consider product bases in $H$ and $L$. Then, there exists a permutation matrix $P$ of size $n^2$ such that
 $$P M(H,L)=M(H_1,E)\otimes M(\overline{H},F).$$ 
\end{teo}
\begin{proof}
Since $H=H_1\otimes_K\overline{H}$, $\{w_i\eta_j\,|\,1\leq i\leq r,\,1\leq j\leq u\}$ is a $K$-basis of $H$. 
As we have seen in \ref{actioninduced} $\rho_H=\rho_{H_1}\otimes\rho_{\overline H}$. 
Considering the fixed bases of $E$ and $F$ we can identify $w_i$ and $\eta_j$ with matrices in $\mathrm{Mat}_r(K)$ and $\mathrm{Mat}_u(K)$ respectively. Then, fixing the $K$-basis $\{\alpha_k z_l\}$ of $L$, the matrix of $w_i\eta_j\in H$ is
$$
w_i\eta_j=w_i\otimes \eta_j
$$
For the block matrix of the action of $H$ we have
$$
\begin{pmatrix}
w_1\otimes \eta_1\\
\dots\\
w_1\otimes \eta_{u}\\
w_2\otimes \eta_1\\
\dots\\
w_{r}\otimes \eta_{u}
\end{pmatrix}=
\begin{pmatrix}
w_1\\
w_2\\
\dots\\
w_{r}
\end{pmatrix}
\circ
\begin{pmatrix}
\eta_1\\
\eta_2\\
\dots\\
\eta_{u}
\end{pmatrix},
$$
which is the Tracy–Singh product of partitioned matrices. To obtain the matrix of the action, we have to take
coordinates with respect to the canonical basis of the respective matrix spaces:
$M(H,L)$ has columns $\varphi(w_i\eta_j)=\varphi_n(w_i\otimes \eta_j)\in K^{n^2}$
while $M(H_1,E)$ has columns 
$\varphi_r(w_i)\in K^{r^2}$ and
$M(\overline H,F)$ has columns 
$\varphi_u(\eta_j)\in K^{u^2}$.

Working with canonical bases of matrices and vectors we find the permutation.
$$
\varphi(E^r_{ab}\otimes E^u_{cd})=\varphi(E^{n}_{u(a-1)+c, u(b-1)+d})=e_{nu(b-1)+n(d-1)+u(a-1)+c}\in K^{n^2}
$$
Here, we use $e_m$ for vectors of
the canonical basis of $K^{n^2}$ and $E^r_{a,b}$ denotes the $r\times r$ matrix having a $1$ in position $(a,b)$ and coefficients $0$ 
elsewhere. Notations for the other matrices is analogous.
On the other hand, if we indicate with superindexes vectors of the canonical basis of $K^{r^2}$ and 
$K^{u^2}$ respectively,
$$\varphi_r(E^r_{ab})\otimes \varphi_u(E^u_{cd})=
e^r_{r(b-1)+a}\otimes e^u_{u(d-1)+c}=
e_{nu(b-1)+u^2(a-1)+u(d-1)+c}
$$
Therefore,  
$P$ is the permutation matrix which places row 
$nu(b-1)+n(d-1)+u(a-1)+c$ in position 
$nu(b-1)+u^2(a-1)+u(d-1)+c$ for $1\le a,b\le r$ and  $1\le c,d\le u$.
\end{proof}

For example, if $r=u=2$, then the matrix $P$ has size $16$ and corresponds to the permutation $(3,5)(4,6)(11,13)(12,14)\in\mathfrak S_{16}$.  If $r=3, u=2$, then the matrix $P$ has size $36$ and corresponds to the
permutation
$(3\ 5\ 9\ 7)(4\ 6\ 10\ 8)(15\ 17\ 21\ 19)(16\ 18\ 22\ 20)(27\ 29\ 33\ 31)(28\ 30\ 34\ 32)\in\mathfrak S_{36}.$

Now we assume the hypothesis of section \ref{sectionintegral}
and we assume that we deal with and induced Hopf Galois structure with $H=H_1\otimes_K\overline{H}$. In this context we will think of $P$ as an unimodular matrix.

We consider the three Hopf actions involved and the corresponding associated orders $\mathfrak{A}_H$, $\mathfrak{A}_{H_1}$ and $\mathfrak{A}_{\overline{H}}$ of $\mathcal{O}_L$, $\mathcal{O}_E$ and $\mathcal{O}_F$, respectively. In what follows we study the relationship between these objects. The main goal is to prove the following:
\begin{teo}\label{teodescassocorder}  If $E/K$ and $F/K$ are arithmetically disjoint, then $$\mathfrak{A}_H=\mathfrak{A}_{H_1}\otimes_{\mathcal{O}_K}\mathfrak{A}_{\overline{H}}.$$
\end{teo}

The idea is to use the  previous method to compute bases of the associated orders involved and prove that the product of suitable bases of $\mathfrak{A}_{H_1}$ and $\mathfrak{A}_{\overline{H}}$ gives a basis for $\mathfrak{A}_H$. 

We fix an $\mathcal{O}_K$-basis $\{\alpha_k\}_{k=1}^{r}$ of $\mathcal{O}_E$ and an $\mathcal{O}_K$-basis $\{z_l\}_{l=1}^{u}$  of $\mathcal{O}_F$. Since we assume that $E/K$ and $F/K$ are arithmetically disjoint, we have  $\mathcal{O}_L=\mathcal{O}_E\otimes_{\mathcal{O}_K}\mathcal{O}_F$, so the products $\{\alpha_kz_l\}_{k,l}$ form an $\mathcal{O}_K$-basis of $\mathcal{O}_L$.
These integral bases are also $K$-bases for the respective fields and Theorem \ref{teomatrixassocorder} applies. 

\begin{pro}
Let us write $M(H,L)=d_MM$ with $d_M\in K$ and $M\in\mathcal{M}_n(\mathcal{O}_K)$ with coprime coefficients. The integral matrices  $\dfrac{1}{d_M}\,  M(H,L)$ and $\dfrac{1}{d_M}(M(H_1,E)\otimes M(\overline{H},F))$ have the same Hermite normal form. 
\end{pro}
\begin{proof}
We apply Theorem \ref{actioninducedmatrices} and observe that since product by $P$ is just a permutation of the rows, the {\it content} $d_M$ does not change. 

We have an unimodular matrix $U\in GL_n(\mathcal{O}_K)$ such that
$$U\, M(H,L)={d_M}\begin{pmatrix}
D\\\hline O\\
\end{pmatrix}$$
where $D$ has coefficients in $\mathcal{O}_K$ and is in Hermite normal form.
Then,
$$UP^{-1}\, PM(H,L)={d_M}\begin{pmatrix}
D\\\hline O\\
\end{pmatrix}$$
Since $UP^{-1}$ is unimodular and the Hermite normal form is unique, we have obtained the Hermite normal form of $\dfrac1{d_M}(M(H_1,E)\otimes M(\overline{H},F))$.
\end{proof}

\begin{pro} If $\dfrac{1}{d_1}M(H_1,E)$ and 
$\dfrac{1}{\overline d}M(\overline H,F)$ have integral coprime coefficients, then $d_M=d_1\overline d$ up to multiplication by an unit and if 
$$
U_1\, M(H_1,E)=d_1\begin{pmatrix}
D_1\\\hline O\\
\end{pmatrix}\qquad \overline U\, M(\overline H,F)={\overline d}\begin{pmatrix}
\ \overline D\ \\ \hline O\\
\end{pmatrix}
$$
with $ U_1\in GL_{r^2}(\mathcal{O}_K)$, $\overline U\in GL_{u^2}(\mathcal{O}_K)$ and $D_1, \overline D$ in Hermite normal form, then
$$
D=D_1\otimes \overline D.
$$
\end{pro}
\begin{proof} The relation between $d_M$, $d_1$ and $\overline{d}$
comes from the fact that the Kronecker product of matrices
with coprime coefficients is also a matrix with coprime coefficients. Since $d_M$, $d_1$ and $\overline{d}$ are determined theirselves up to multiplication by an unit, the equality also holds up to multiplication by an unit.
We have 
$$(U_1 M(H_1,E))\otimes(\overline U M(\overline H,F))=
(U_1\otimes \overline U )\left( 
M(H_1,E)\otimes M(\overline H,F)\right)
$$
by the mixed product property of the Kronecker product.
Therefore,
$$
{d_1}\begin{pmatrix}
D_1\\\hline O\\
\end{pmatrix}\otimes {\overline d}\begin{pmatrix}
\ \overline D\ \\ \hline O\\
\end{pmatrix}
=(U_1\otimes \overline U )P M(H,L) 
$$
$$
{d_M}\begin{pmatrix}
D_1\otimes \overline D\\\hline O\\
\end{pmatrix}
=(U_1\otimes \overline U )P M(H,L) 
$$
Since $(U_1\otimes \overline U )P\in \GL_{n^2}(\mathcal{O}_K)$
and $D_1\otimes \overline D$ is in Hermite normal form, unicity gives $D=D_1\otimes \overline D.$
\end{proof}

\begin{coro} If for a matrix $U_1\in GL_{r^2}(\mathcal{O}_K)$ we have 
$$
U_1\, M(H_1,E)=\begin{pmatrix}
D_1\\\hline O\\
\end{pmatrix}
$$
with $D_1\in\GL_r(K)$, and for a matrix $\overline U\in GL_{u^2}(\mathcal{O}_K)$ we have
$$
\overline U\, M(\overline H,F)=\begin{pmatrix}
\ \overline D\ \\ \hline O\\
\end{pmatrix}
$$
with $\overline D\in\GL_u(K)$, then
$$(U_1\otimes \overline U )P M(H,L)= \begin{pmatrix}
D_1\otimes \overline D\\\hline O\\
\end{pmatrix}$$
with $(U_1\otimes \overline U )P\in \GL_{n^2}(\mathcal{O}_K)$ and 
$D=D_1\otimes \overline D\in \GL_n(K)$.
\end{coro}

Either in Hermite normal form or in general form, since 
$$(D_1\otimes \overline D)^{-1}=D_1^{-1}\otimes  \overline D^{\,-1}
$$
and columns of the corresponding matrices provide bases for 
$\mathfrak{A}_H$, $\mathfrak{A}_{H_1}$ and $\mathfrak{A}_{\overline H}$ respectively,
this finishes the proof of Theorem \ref{teodescassocorder}.

\begin{example}\label{exeigQ3'matrix}
In example \ref{exeigQ3'}, the normal closure of $E$ is $L=\Q(\alpha,z)$, where $z=\sqrt{-1}$. In analogy with Example \ref{degree6}, if we call $F=\Q(z)$, then $L=EF$ and we can induce a Hopf Galois structure $H$ of $L/K$ from the Hopf Galois structure $H_1$ of $E/\Q$ and $\overline{H}$ of $F/\Q$.

Since $F/\Q$ is unramified, the extensions $E/\Q$ and $F/\Q$ are arithmetically disjoint and
we can obtain the associated order to the induced Hopf Galois structure of $L/\Q$ using the above method. We take
$$
D_1=\begin{pmatrix}
1&0&-1\\0&1&0\\0&0&3
\end{pmatrix},\qquad
\overline D=
\begin{pmatrix}
 1&-1\\ 0&2
 \end{pmatrix}.
$$
\begin{equation*}\label{table4bis}
\begin{tabular}{c|ccc} $\mathfrak{A}_{ H_1}$& \begin{tabular}[c]{@{}l@{}}$1$\end{tabular} & 
\begin{tabular}[c]{@{}l@{}}$\alpha$\end{tabular} &
\begin{tabular}[c]{@{}l@{}}$\alpha^2$\end{tabular}
\\ \hline
$w_1$ & $1$ & $\alpha$ & $\alpha^2$ \\[1ex]
$w_2$& $0$& $3+9\alpha+2\alpha^2$ & $-3-30\alpha-9\alpha^2$ \\[1ex]
$w_3'=\dfrac{w_1+w_3}{3}$ & $1$ & $-1$& $3$ \\
\end{tabular}
\qquad
\begin{tabular}{c|cc}
$\mathfrak{A}_{\overline H}$ &$1$&$z$\\ \hline
$\eta_1$ & $1$ & $z$ \\
$\eta_2'=\dfrac{\eta_1+\eta_2}{2}$ & $1$ & $0$ \\
\end{tabular}
\end{equation*}
Then,
$$
D=D_1\otimes \overline D=\begin {pmatrix} 
 1&-1&0&0&-1&1\\
 0&2&0&0&0&-2\\
 0&0&1&-1&0&0\\ 
 0&0&0&2&0&0\\ 
0&0&0&0&3&-3\\ 
0&0&0&0&0&6
\end{pmatrix}\quad 
D^{-1}=
\begin{pmatrix} 
 1&1/2&0&0&1/3&1/6\\ 
 0&1/2&0&0&0&1/6\\ 
0&0&1&1/2&0&0\\ 
0
&0&0&1/2&0&0\\ 0&0&0&0&1/3&1/6\\ 
0&0&0&0&0&1/6
\end{pmatrix} 
$$
gives
$$
\mathfrak{A}_H=\mathbb{Z}_3\left[\, 
w_1\eta_1,
\dfrac{w_1\eta_1+w_1\eta_2}2,
w_2\eta_1,
\dfrac{w_2\eta_1
+w_2\eta_2}{2},
\dfrac{w_1\eta_1+w_3\eta_1}3,
\dfrac{w_1\eta_1+w_1\eta_2+w_3\eta_1+w_3\eta_2}{6}\,\right]
$$
$$
\mathfrak{A}_H=\mathbb{Z}_3\left[\, w_1\eta_1,w_1\eta_2',
w_2\eta_1,
w_2\eta_2',
w_3'\eta_1,
w_3'\eta_2'
\,\right].
$$
Taking the basis $1,z,\alpha,\alpha z,\alpha_2,\alpha^2z$ for $\mathcal{O}_L$ we can compute the action of $\mathfrak{A}_H$ using the tables above.
\end{example}

\subsection{Induced Hopf Galois module structure}


 Our aim to relate the freeness of $\mathcal{O}_L$ as $\mathfrak{A}_H$-module with the freeness of $\mathcal{O}_E$ as $\mathfrak{A}_{H_1}$-module and the freeness of $\mathcal{O}_F$ as $\mathfrak{A}_{\overline{H}}$-module.

\begin{teo}\label{teoinducedfreeness} Let us assume that $E/K$ and $F/K$ are arithmetically disjoint. If $\mathcal{O}_E$ is $\mathfrak{A}_{H_1}$-free and $\mathcal{O}_F$ is $\mathfrak{A}_{\overline{H}}$-free, then $\mathcal{O}_L$ is $\mathfrak{A}_H$-free. 

Moreover, if $\gamma$ is a $\mathfrak{A}_{H_1}$-free generator of $\mathcal{O}_E$ and $\delta$ is a $\mathfrak{A}_{\overline{H}}$-free generator of $\mathcal{O}_F$, then $\gamma\delta$ is a $\mathfrak{A}_H$-free generator of $\mathcal{O}_L$.
\end{teo}
\begin{proof}
Let $\{v_i\}_{i=1}^{r}$ be an $\mathcal{O}_K$-basis of $\mathfrak{A}_{H_1}$ and let $\{\mu_j\}_{j=1}^{u}$ be an $\mathcal{O}_K$-basis of $\mathfrak{A}_{\overline{H}}$. Then, $\{v_i\cdot\gamma\}_{i=1}^{r}$ is an $\mathcal{O}_K$-basis of $\mathcal{O}_E$ and $\{\mu_j\cdot\delta\}_{j=1}^{u}$ is an $\mathcal{O}_K$-basis of $\mathcal{O}_F$. Since $\mathcal{O}_L=\mathcal{O}_E\otimes_{\mathcal{O}_K}\mathcal{O}_F$, the product of these bases is an $\mathcal{O}_K$-basis of $\mathcal{O}_L$. But that basis is formed by the elements $$(v_i\cdot\gamma)(\mu_j\cdot\delta)=(v_i\mu_j)\cdot(\gamma\delta),\,1\leq i\leq r,\,0\leq j\leq u.$$

Since $\mathfrak{A}_H=\mathfrak{A}_{H_1}\otimes_{\mathcal{O}_K}\mathfrak{A}_{\overline{H}}$, this amounts to say that $\gamma\delta$ is a $\mathfrak{A}_H$-free generator of $\mathcal{O}_L$.
\end{proof}

\begin{example} Let us consider once again the situation of Example \ref{exeigQ3'}.  We know by Example \ref{exeigQ3'matrix} that the action of the associated order over $\mathcal{O}_E$ is given by 
\begin{equation}\label{table4bis}
\begin{tabular}{c|ccc} $\mathfrak{A}_{ H_1}$& \begin{tabular}[c]{@{}l@{}}$1$\end{tabular} & 
\begin{tabular}[c]{@{}l@{}}$\alpha$\end{tabular} &
\begin{tabular}[c]{@{}l@{}}$\alpha^2$\end{tabular}
\\ \hline
$w_1$ & $1$ & $\alpha$ & $\alpha^2$ \\[1ex]
$w_2$& $0$& $3+9\alpha+2\alpha^2$ & $-3-30\alpha-9\alpha^2$ \\[1ex]
$w_3'$ & $0$ & $-1$ & $3$ \\
\end{tabular}
\end{equation}
Hence, for an element $\beta=\beta_0+\beta_1\alpha+\beta_2\alpha^2\in\mathcal{O}_E$, we have: $$w_1\cdot\beta=\beta_0+\beta_1\alpha+\beta_2\alpha^2,$$ $$w_2\cdot\beta=3(\beta_1-\beta_2)+(9\beta_1-30\beta_2)\alpha+(2\beta_1-9\beta_2)\alpha^2$$
$$w_3'\cdot\beta=\beta_0-\beta_1+3\beta_2.$$
If we take $\beta=\alpha$, the associated matrix is $$\mathcal{D}_{\alpha}(H_1,E)=\begin{pmatrix}
0 & 3 & -1\\
1 & 9 & 0 \\
0& 2 &  0
\end{pmatrix},$$
which has determinant $-2$. Then, $\alpha$ is a generator of $\mathcal{O}_E$ as $\mathfrak{A}_{H_1}$-module and since $F/\Q$ is quadratic, by Example \ref{quadraticfreeness}, $1+z$ is an $\mathfrak{A}_{\overline{H}}$-generator of $\mathcal{O}_F$. We know by Example \ref{exeigQ3'matrix} that $E/\Q$ and $F/\Q$ are arithmetically disjoint. Then, we can apply Theorem \ref{teoinducedfreeness} to obtain that $\mathcal{O}_L$ is $\mathfrak{A}_H$-free with generator $\alpha(1+z)$.
\end{example}

\subsection{Freeness after tensoring by $\mathcal{O}_F$}

We finally compute the associated order of $\mathcal{O}_L$ in $H_1\otimes_K F$ and discuss the freeness of $\mathcal{O}_L$ in $\mathfrak{A}_{H_1\otimes_KF}$. Note that this is actually a Hopf Galois structure of $L/F$ because $H_1$ is a Hopf Galois structure of $E/K$ and $F$ is $K$-flat. Moreover, the action of $H_1\otimes_K F$ on $L$ is obtained by extending $F$-linearly the one of $H_1$ on $E$.

We study the relation between $\mathfrak{A}_{H_1}$ and $\mathfrak{A}_{H_1\otimes_K F}$, as well as the $\mathfrak{A}_{H_1}$-freeness of $\mathcal{O}_E$ and the $\mathfrak{A}_{H_1\otimes_K F}$-freeness of $\mathcal{O}_L$. In order to do this, we need a suitable description of elements of $\mathcal{O}_L$. For this reason, we make again the hypothesis that $E/K$ and $F/K$ are arithmetically disjoint, which implies that  $\mathcal{O}_L=\mathcal{O}_E\otimes_{\mathcal{O}_K}\mathcal{O}_F$.

Let $\{\alpha_i\}_{i=1}^{r}$ be an $\mathcal{O}_K$-basis of $\mathcal{O}_E$ and let $\{z_j\}_{j=1}^{u}$ be an $\mathcal{O}_K$-basis of $\mathcal{O}_F$. Since $\mathcal{O}_L=\mathcal{O}_E\otimes_{\mathcal{O}_K}\mathcal{O}_F$, $\{\alpha_iz_j\}_{i,j}$ is an $\mathcal{O}_K$-basis of $\mathcal{O}_L$.

\begin{pro} If $E/K$ and $F/K$ are arithmetically disjoint, then $\mathfrak{A}_{H_1\otimes_K F}=\mathfrak{A}_{H_1}\otimes_{\mathcal{O}_K}\mathcal{O}_F$.
\end{pro}
\begin{proof}
Let $H^{(1)}=H_1\otimes_K F$. First, we prove that $\mathfrak{A}_{H_1}\otimes_{\mathcal{O}_K}\mathcal{O}_F\subset\mathfrak{A}_{H^{(1)}}$. It is clearly contained in $H_1\otimes_{\mathcal{O}_K}F=H_1\otimes_K F=H^{(1)}$. On the other hand, it acts $\mathcal{O}_K$-linearly on $\mathcal{O}_L$ componentwise since $\mathcal{O}_L=\mathcal{O}_E\otimes_{\mathcal{O}_K}\mathcal{O}_F$. This proves the claim. 

For the reverse inclusion, let $h\in\mathfrak{A}_{H^{(1)}}$. Trivially, $h\in H^{(1)}=H_1\otimes_KF$. Since $\{z_j\}_{j=1}^{u}$ is a $K$-basis of $F$ and $H_1$ is $K$-flat, it is also an $H_1$-basis of $H^{(1)}$. Then, $$h=\sum_{j=1}^{u}h^{(j)}z_j,\,h^{(j)}\in H_1.$$ 
The result will follow from the fact that $h^{(j)}\in\mathfrak{A}_{H_1}$ for all $1\leq j\leq u$. In order to prove this, we may check that $h^{(j)}\cdot\gamma\in\mathcal{O}_E$ for all $\gamma\in\mathcal{O}_E$. Take any such $\gamma\in\mathcal{O}_E$. In particular $\gamma\in\mathcal{O}_L$, and since $h\in\mathfrak{A}_{H^{(1)}}$, we have that $h\cdot_L\gamma\in\mathcal{O}_L$. But $$h\cdot_L\gamma=\left(\sum_{j=1}^{u}h^{(j)}z_j\right)\cdot_L\gamma=\sum_{j=1}^{u}(h^{(j)}\cdot_E\gamma)z_j\in\mathcal{O}_L.$$ Now, $\{z_j\}_{j=1}^{u}$ is an $\mathcal{O}_E$-basis of $\mathcal{O}_L$ because $\mathcal{O}_E$ is $\mathcal{O}_K$-flat. Hence, the previous expression yields that $h^{(j)}\cdot_E\gamma\in\mathcal{O}_E$ for all $1\leq j\leq u$.
\end{proof}

With this result, we have determined how the associated order changes when we tensor with an arithmetically disjoint extension. Now, we move to the question of the freeness of $\mathcal{O}_L$ as $\mathfrak{A}_{H_1\otimes_K F}$-module. We will prove $2$ of Theorem \ref{thirdmaintheorem}.

\begin{coro} Assume that $E/K$ and $F/K$ are arithmetically disjoint. If $\mathcal{O}_E$ is $\mathfrak{A}_{H_1}$-free, then $\mathcal{O}_L$ is $\mathfrak{A}_{H_1\otimes_K F}$-free.
\end{coro}
\begin{proof}
Since $\mathcal{O}_F$ is $\mathcal{O}_K$-flat, $\mathcal{O}_E\otimes_{\mathcal{O}_K}\mathcal{O}_F=\mathcal{O}_L$ is $\mathfrak{A}_{H_1}\otimes_{\mathcal{O}_K}\mathcal{O}_F$-free. By the previous result, $\mathfrak{A}_{H_1}\otimes_{\mathcal{O}_K}\mathcal{O}_F=\mathfrak{A}_{H_1\otimes_K F}$ and the claim follows.
\end{proof}

\Addresses

\end{document}